\documentclass[12pt,leqno]{article}
\usepackage{amsfonts}
\usepackage{amsmath, amsthm, amsfonts, amssymb, color}
\usepackage{mathrsfs}
\usepackage{color}
\newtheorem{thm}{Theorem}[section]
\newtheorem{cor}[thm]{Corollary}
\newtheorem{lem}[thm]{Lemma}

\theoremstyle{definition}

\newcommand{\scr}[1]{\mathscr #1}
\definecolor{wco}{rgb}{0.5,0.2,0.3}

\numberwithin{equation}{section} \theoremstyle{remark}

\newcommand{\ua}{\uparrow}

\title{{\bf    Integration by Parts formula for SPDEs with Multiplicative Noise and its Applications}\footnote{Supported in
 part by  NNSFC(11131003, 11431014), the 985 project and the Laboratory of Mathematical and  Complex Systems.} }
\author{
{\bf    Xing Huang $^{a)}$, Shao-Qin Zhang $^{b)}$, Li-Xia Liu $^{a)}$ }\\
\footnotesize{ a)School of Mathematical Sciences,
Beijing Normal
University, Beijing 100875, China,}\\
\footnotesize{  XingHuang@mail.bnu.edu.cn, lixialiu1@126.com}\\
 \footnotesize{ b)School of Statistics and Mathematics, Central University of Finance and Economics, Beijing 100081, China, }\\
\footnotesize{  zhangsq@cufe.edu.cn }}
\begin{document}
\allowdisplaybreaks
\def\R{\mathbb R}  \def\ff{\frac} \def\ss{\sqrt} \def\B{\mathbf
B}
\def\N{\mathbb N} \def\kk{\kappa} \def\m{{\bf m}}
\def\ee{\varepsilon}\def\ddd{D^*}
\def\dd{\delta} \def\DD{\Delta} \def\vv{\varepsilon} \def\rr{\rho}
\def\<{\langle} \def\>{\rangle} \def\GG{\Gamma} \def\gg{\gamma}
  \def\nn{\nabla} \def\pp{\partial} \def\E{\mathbb E}
\def\d{\text{\rm{d}}} \def\bb{\beta} \def\aa{\alpha} \def\D{\scr D}
  \def\si{\sigma} \def\ess{\text{\rm{ess}}}
\def\beg{\begin} \def\beq{\begin{equation}}  \def\F{\scr F}
\def\Ric{\text{\rm{Ric}}} \def\Hess{\text{\rm{Hess}}}
\def\e{\text{\rm{e}}} \def\ua{\underline a} \def\OO{\Omega}  \def\oo{\omega}
 \def\tt{\tilde} \def\Ric{\text{\rm{Ric}}}
\def\cut{\text{\rm{cut}}} \def\P{\mathbb P} \def\ifn{I_n(f^{\bigotimes n})}
\def\C{\scr C}      \def\aaa{\mathbf{r}}     \def\r{r}
\def\gap{\text{\rm{gap}}} \def\prr{\pi_{{\bf m},\varrho}}  \def\r{\mathbf r}
\def\Z{\mathbb Z} \def\vrr{\varrho} \def\ll{\lambda}
\def\L{\scr L}\def\Tt{\tt} \def\TT{\tt}\def\II{\mathbb I}\def\X{\scr X}
\def\i{{\rm in}}\def\Sect{{\rm Sect}}  \def\H{\mathbb H}
\def\M{\scr M}\def\Q{\mathbb Q} \def\texto{\text{o}} \def\LL{\Lambda}
\def\Rank{{\rm Rank}} \def\B{\scr B} \def\i{{\rm i}} \def\HR{\hat{\R}^d}
\def\to{\rightarrow}\def\l{\ell}\def\iint{\int}
\def\EE{\scr E}\def\no{\nonumber}
\def\A{\scr A}\def\V{\mathbb V}\def\osc{{\rm osc}}\def\H{\scr H}
\def\BB{\scr B}\def\Ent{{\rm Ent}}

\maketitle

\begin{abstract} By using the Malliavin calculus, the Driver-type integration by parts formula is established for the semigroup associated to to SPDEs with Multiplicative Noise. Moreover, estimates on the density of heat kernel w.r.t. Lebesgue measure are obtained in finite dimension case.
\end{abstract} \noindent
 \vskip 2cm

\section{Introduction}
A significant application of the Malliavin calculus (\cite{M,N}) is to describe the density of a Wiener
functional using the integration by parts formula. In 1997, Driver \cite{D} established the
following integration by parts formula for the heat semigroup $P_t$ on a compact Riemannian
manifold $M$:
\begin{equation}\label{1.1}
P_t(\nabla_Z f)=\mathbb{E}(f(X_t)N_t),\ \ f\in C^1(M), Z\in\X,
\end{equation}
where $\X$ is is the set of all smooth vector fields on $M$, and $N_t$ is a random variable depending
on $Z$ and the curvature tensor. This formula has many applications. For example, we are able to characterize the derivative w.r.t. the second variable $y$ of the heat kernel $p_t(x, y)$, see \cite{W} for a study on integration by parts formulas and applications for SDEs and SPDEs driven by Wiener processes. Moreover, if $N_t$ is exponential integrable, \eqref{1.1} implies the shift Harnack inequality, see also \cite{W} for details.

So far, there are many results on the Driver-type integration by parts formula for SDEs or SPDEs.
The backward coupling method developed in \cite{W} has been used in \cite{F,SQZ} for SDEs driven by fractional Brownian motions and SPDEs driven by Wiener processes

Recently, using a finite many jumps approximation and Malliavin calculus, Wang obtains integration by parts formula for SDEs and SPDEs with additive noise driven by subordinated Brownian motion, see \cite{W1,W2}.

However, all the above results are considered in additive noise case. The aim of this paper is to derive the integration by parts formula for SPDEs with multiplicative noise by Malliavin calculus and to
derive estimates on the derivatives of heat kernel.

The main difficulty in obtaining the integration by parts formula is to give a representation of $D_{h_k}J_T$ (see the proof of Theorem \ref{T2.2}). Unfortunately, in multiplicative noise case, the equations for $D_{h_k}J_T$ is so sophisticated that the Duhamel's formula used in the additive noise case is unavailable. Instead, we applying Lemma \ref{L3.1} which is crucial in the proof of the main results.

Consider the following SPDE on a separable Hilbert spaces $(\mathbb{H},\langle,\rangle,|\cdot|)$ :
\beq \label{1.2}
\d X_{t}= AX_{t}\d t+b_{t}(X_{t})\d t+\sigma_{t}(X_{t})\d W_{t},\ \ X_0=x\in\mathbb{H},
\end{equation}
where $b: [0,\infty)\times \mathbb{H}\to \mathbb{H}$ are measurable locally bounded (i.e. bounded on bounded sets), and $\sigma: [0,\infty)\times \mathbb{H}\to \L(\mathbb{H})$ is measurable, where $\L(\mathbb{H})$ is the space of bounded linear operators on $\mathbb{H}$ equipped with the operator norm $\|\cdot\|$. Moreover,

\beg{enumerate}
\item[(i)] $W$ is a cylindrical Brownian motion on $\mathbb{H}$ with respect to a complete filtration  probability space $(\OO, \F, \{\F_{t}\}_{t\ge 0}, \P)$. More precisely, $W=\sum_{n=1}^{\infty}{w^{n}e_{n}}$ for a sequence of independent one dimensional Brownian motions $\{w^{n}\}_{n\geq 1}$ with respect to $(\OO, \F,
\{\F_{t}\}_{t\ge 0}, \P)$, where $\{e_{n}\}_{n\geq 1}$ is an orthonormal basis on $\mathbb{H}$.

\item[(ii)] $(A,\D(A))$ is a linear operator generating a $C_{0}$-contraction semigroup $\e^{At}$ such that
\beg{equation}\beg{split}\label{1.3}
\int_{0}^{T}\|\e^{At}\|_{\mathrm{HS}}^{2}\d t<\infty, \ \ T>0,
\end{split}\end{equation}
where $\|\cdot\|_{\mathrm{HS}}$ is the Hilbert-Schmidt norm.

\item[(iii)] There exists a non-negative function $K\in C([0,\infty),[0,\infty))$ such that
\beg{equation}\beg{split}\label{1.4}
\|\nabla_v b_{s}(x)\|\vee\|\nabla_v \sigma_{s}(x)\|_{\mathrm{HS}}\leq K(s)|v|, \ \ s\geq0, x, v\in\mathbb{H}.
\end{split}\end{equation}
\end{enumerate}
Then the equation \eqref{1.2} has a unique mild solution $X_{t}(x)$, and the associated Markov semigroup $P_{t}$ is defined as follows:
\beg{equation*}
P_{t}f(x):=\mathbb{E}f(X_{t}(x)),\ \  f\in \B_{b}(\mathbb{H}), t\geq 0, x\in\mathbb{H}.
\end{equation*}
Since for any $t\geq 0$, $\mathrm{Ker}(\e^{At})={0}$, the inverse operator $\e^{-At}:\mathrm{Im}(\e^{At})\to\mathbb{H}$ is well defined. To establish the integration by parts formula, we need the following assumptions:
\beg{enumerate}
\item[(H1)] For any $(t,x)\in [0,\infty)\times\mathbb{H}$, $b_{t}, \sigma_{t}\in C^{2}(\mathbb{H})$, and  there holds $\nabla b_{t}(x):\mathrm{Im}(\e^{At})\to \mathrm{Im}(\e^{At})$, $\nabla \sigma_{t}(x):\mathrm{Im}(\e^{At})\to \L(\mathbb{H},\mathrm{Im}(\e^{At}))$.
    Let
\beg{equation}\label{1.5}\beg{split}
&B_{t}(x):=\e^{-At}\nabla b_{t}(x)\e^{At},\ \ \Sigma_{t}(x):=\e^{-At}\nabla \sigma_{t}(x)\e^{At},\\
&\sigma_{t}^{(k)}(x):=\sigma_{t}(x)e_{k},\ \ \Sigma_{t}^{(k)}(x):=\e^{-At}\nabla \sigma_{t}^{(k)}(x)\e^{At},
\ \ k\geq 1,
(t,x)\in [0,\infty)\times\mathbb{H}.
\end{split}\end{equation}
    Assume
\beg{equation}\label{1.6}
\left\|B_{t}(x)\right\|\vee\left\{\sum_{k=1}^{\infty} \left\|\Sigma_{t}^{(k)}(x)\right\|^2\right\}^{\frac{1}{2}} \vee\sum_{k=1}^{\infty}\left\|\Sigma_{t}^{(k)}(x)\right\|\leq K_{1}(t), \ \ t\geq0, x\in\mathbb{H},
\end{equation}
and
\beg{equation}\label{1.7}
\left\|\nabla B_{t}(x)\right\|\vee\left\{\sum_{k=1}^{\infty}\left\|\nabla\Sigma_{t}^{(k)}(x)\right\| ^2\right\}^{\frac{1}{2}}\leq K_{2}(t),\ \  t\geq0, x\in\mathbb{H}
\end{equation}
for two increasing functions $K_{1}$, $K_{2}\in C([0,\infty),[0,\infty))$.

\item[(H2)] $\sigma$ is invertible. Moreover, there exist a strictly positive increasing function $\lambda$, $\lambda\in C([0,\infty),(0,\infty))$ such that
\beg{equation}\label{1.8}
\left\|\sigma^{-1}_{t}(x)\right\|\leq\lambda(t),\ \  t\geq0,x\in\mathbb{H}.
\end{equation}
\end{enumerate}

(H2) is a standard non-degenerate assumption, while (H1) comes from \cite{W2}, which means that the interaction between far away directions are weak enough. For example, let $0< \lambda_{1}\leq \lambda_{2}\leq\cdots\cdots$ be all eigenvalues of $-A$ counting multiplicities, and $\{e_{k}\}_{k\geq 1}$ are the corresponding eigenbasis. (H1) holds if
\beg{equation*}\beg{split}
&|\langle \nabla_{e_i}b_t, e_j\rangle|\leq K_1(t)\e^{-t|\lambda_i-\lambda_j|}, \\
&|\langle \nabla\nabla_{e_i}b_t, e_j\rangle|\leq K_2(t)\e^{-t|\lambda_i-\lambda_j|},\\
&\sum_{k=1}^{\infty}|\langle \nabla_{e_i}\sigma^{k}_t, e_j\rangle|^2\leq K^2_1(t)\e^{-2t|\lambda_i-\lambda_j|}, \\
&\sum_{k=1}^{\infty}|\langle \nabla\nabla_{e_i}\sigma^{k}_t, e_j\rangle|^2\leq K^2_2(t)\e^{-2t|\lambda_i-\lambda_j|},
\ \ t\geq 0, i, j\geq 1.
\end{split}\end{equation*}

In addition, for simplicity, set
\beg{equation}\label{1.9}\beg{split}
&(\nabla B_{t})(x)(u,v):=\nabla_v \left(B_{t}(\cdot)u\right)(x),\\
&\left(\nabla\Sigma_{t}^{(k)}\right)(x)(u,v):=\nabla_v \left(\Sigma_{t}^{(k)}(\cdot)u\right)(x),
\ \ k\geq 1,
(t,x,u,v)\in [0,\infty)\times\mathbb{H}^3.
\end{split}\end{equation}
\section{Main results}
To state our main results, for any $s\geq0$, we introduce $\L(\mathbb{H})$-valued processes $(J_{s,t})_{t\geq s}$ and $(J^{A}_{s,t})_{t\geq s}$, which solve the following SDEs:
\beq\label{2.1}\beg{split}
\d J_{s,t}=B_{t}(X_{t})J_{s,t}\d t+\sum_{k=1}^{\infty}\Sigma_{t}^{(k)}(X_{t})J_{s,t}\d w^{k}_{t}, \ \ J_{s,s}=I
\end{split}\end{equation}
and
\beq\label{2.2}\beg{split}
&\d J^{A}_{s,t}=(A+\nabla b_{t}(X_{t}))J^{A}_{s,t}\d t+\sum_{k=1}^{\infty}\nabla\sigma_{t}^{(k)}(X_{t})J^{A}_{s,t}\d w^{k}_{t}, \ \ J^{A}_{s,s}=I.
\end{split}\end{equation}
From \eqref{1.7}, \eqref{2.1} and \eqref{2.2} are well defined. Set $J_{t}=J_{0,t}$ and $J^{A}_{t}=J^{A}_{0,t}$, then $J_t^{A}=\e^{At}J_t$.

Firstly, we use It\^{o}'s formula to derive the equation for $\{J_{t}^{-1}\}_{t\geq 0}$.

Assume
\beg{equation*}
\d J_{t}^{-1}=G_t\d t+\sum_{k=1}^{\infty}H^{(k)}_t\d  w^k_{t}, \ \ J_{0}^{-1}=I,
\end{equation*}
where $G$ and $H$ are to be determined.

Applying It\^{o}'s formula (see \cite[section 23]{H}), combining with \eqref{2.1}, we have
\beg{equation*}\beg{split}
\d J_{t}J_{t}^{-1}&= B_{t}( X_{t})J_{t}J_{t}^{-1}\d t+\sum_{k=1}^{\infty}\Sigma^{(k)}_{t}( X_{t})J_{t}J_{t}^{-1}\d  w^k_{t}\\
&+J_{t}G_t\d t+\sum_{k=1}^{\infty}J_{t}H^{(k)}_t\d  w^k_{t}+\sum_{k=1}^{\infty}\Sigma^{(k)}_{t}( X_{t})J_{t}H^{(k)}_t\d t.
\end{split}\end{equation*}
Since $\d J_{t}J_{t}^{-1}=0$, it holds that
\beg{equation*}\beg{split}
&G_t=-J_{t}^{-1}\left\{ B_{t}( X_{t})-\sum_{k=1}^{\infty} \left(\Sigma^{(k)}_{t}( X_{t})\right)^{2}\right\}, \quad H^{k}_t=-J_{t}^{-1}\Sigma^{(k)}_{t}( X_{t}), \ \ k\geq 1.
\end{split}\end{equation*}
Thus, we obtain
\beg{equation}\label{2.3}\beg{split}
\d J_{t}^{-1}&=-J_{t}^{-1}\left\{ B_{t}( X_{t})-\sum_{k=1}^{\infty} \left(\Sigma^{(k)}_{t}( X_{t})\right)^{2}\right\}\d t-\sum_{k=1}^{\infty}J_{t}^{-1}\Sigma^{(k)}_{t}( X_{t})\d  w^k_{t},\quad J_{0}^{-1}=I.
\end{split}\end{equation}
\paragraph {Remark 2.1} Since the inverse of $J_t^A$ does not exist in infinite dimension, we will use $J_t^{-1}$ to construct $h$ in \eqref{3.5} in stead of $(J_t^A)^{-1}$ in finite dimension, seeing  details in the proof of Theorem 2.1. Moreover, to ensure the existence of $J_t^{-1}$, we assume (H1).

We have the following estimates for $J_t$ and $J_t^{-1}$.
\begin{lem}\label{L2.1} Assume (H1) and (H2). Then for any $x\in\mathbb{H}$, $t\geq 0$, $p\geq 2$, it holds that
\beg{equation}\label{2.4}\beg{split}
\sup_{s\in[0,t]}\mathbb{E}\|J_{s}\|^{p}\leq 3^{p-1}\exp\left\{3^{p-1}\left(t^{p-1}+t^{\frac{p}{2}-1}\right)K^{p}_{1}(t)\right\}
\end{split}\end{equation}
and
\beg{equation}\label{2.5}\beg{split}
\sup_{s\in[0,t]}\mathbb{E}\|J_{s}^{-1}\|^{p}\leq 3^{p-1}\exp\left\{3^{p-1} \left[t^{p-1}\left(K_{1}(t)+K^{2}_{1}(t)\right)^{p}+t^{\frac{p}{2}-1}K_{1}^{p}(t)\right]\right\}.
\end{split}\end{equation}
\end{lem}
\begin{proof} By Burkerholder-Davis-Gundy inequality, it follows from \eqref{2.1} that
\beg{equation*}\beg{split}
\mathbb{E}\left\|J_{t}\right\|^{p}&\leq 3^{p-1}+3^{p-1}\mathbb{E}\left\|\int_{0}^{t}B_s(X_s)J_{s}\d s\right\|^{p}+3^{p-1}\mathbb{E}\left\|\int_{0}^{t}\sum_{k=1}^{\infty}\Sigma_{s}^{(k)}(X_{s})J_{s}\d w^{k}_{s}\right\|^{p}\\
&\leq 3^{p-1}+3^{p-1}t^{p-1}\mathbb{E}\int_{0}^{t}\|B_s(X_s)\|^p\|J_{s}\|^{p}\d s\\
&+3^{p-1}t^{\frac{p}{2}-1}\mathbb{E}\int_{0}^{t}\left(\sum_{k=1}^{\infty} \|\Sigma_{s}^{(k)}(X_{s})\|^2\right)^{\frac{p}{2}}\|J_{s}\|^p\d s\\
&\leq 3^{p-1}+3^{p-1}\left[t^{p-1}K^{p}_{1}(t)+t^{\frac{p}{2}-1}K_{1}^{p}(t)\right] \int_{0}^{t}\mathbb{E}\|J_{s}\|^{p}\d s.
\end{split}\end{equation*}
Applying Gronwall inequality, we have
\beg{equation}\label{2.6}\beg{split}
&\sup_{s\in[0,t]}\mathbb{E}\left\|J_{s}\right\|^{p}\leq 3^{p-1}\exp\left\{3^{p-1}\left(t^{p-1}+t^{\frac{p}{2}-1}\right)K^{p}_{1}(t)\right\}.
\end{split}\end{equation}
Similarly, \eqref{2.3} yields that
\beg{equation}\label{2.7}\beg{split}
&\sup_{s\in[0,t]}\mathbb{E}\left\|J_{s}^{-1}\right\|^{p}\leq 3^{p-1}\exp\left\{3^{p-1} \left[t^{p-1}\left(K_{1}(t)+K^{2}_{1}(t)\right)^{p}+t^{\frac{p}{2}-1}K_{1}^{p}(t)\right]\right\}.
\end{split}\end{equation}
Thus, we finish the proof.
\end{proof}
The main result is the following.
\beg{thm}\label{T2.2} Assume (H1) and (H2), then the integration formula by parts holds, i.e.
\beq\label{2.8}\beg{split}
P_{T}(\nabla_{\e^{AT}v}f)=\frac{1}{T}\mathbb{E}\{f(X_{T})M_{T}^{v}\}, \ \ v\in\mathbb{H}, f\in C_{b}^{1}(\mathbb{H})
\end{split}\end{equation}
holds for
\beg{equation}\label{2.9}\beg{split}
M_{T}^{v}:&=\left\langle\int_{0}^{T}\left[\sigma^{-1}_t(X_{t})J^{A}_{t}\right]^{\ast}\d W_{t},J_{T}^{-1}v\right\rangle\\
&+\int_{0}^{T}\mathrm{Tr}\left\{t J_{t}^{-1}\left[
(\nabla B_{t})(X_{t})\left(J_{t}J_{T}^{-1}v,J^{A}_{t}\right)\right]\right\}\d t\\
&+\sum_{k=1}^{\infty}\left\langle\left[\int_{0}^{T}J_{t}^{-1}\sum_{j=1}^{\infty}\left(\nabla_{tJ_{t}^{A}e_{k}} \Sigma^{(j)}_{t}\right)(X_{t})J_{t}\d w^{j}_{t}\right]J_{T}^{-1}v, e_{k}\right\rangle\\
&+\int_{0}^{T}\mathrm{Tr}\left\{J_{t}^{-1}\left[\Sigma_{t}(X_{t})J_{t}J_{T}^{-1}v\right]\left(\sigma_t^{\ast} (\sigma_{t}\sigma_t^{\ast})^{-1}\right)(X_{t})J^{A}_{t} \right\}\d t\\
&-\int_{0}^{T}\mathrm{Tr}\left\{tJ_{t}^{-1}\sum_{j=1}^{\infty}\Sigma^{(j)}_{t}(X_{t})\left[\left(\nabla \Sigma^{(j)}_{t}\right)(X_{t})\left(J_{t}J_{T}^{-1}v,J^{A}_{t}\right)\right]\right\}\d t.
\end{split}\end{equation}
\end{thm}
\paragraph {Remark 2.2} Every term in \eqref{2.9} is well defined by \eqref{1.3}, (H1), (H2), Lemma \ref{L2.1}.

This result extends \cite[Theorem 5.1]{W} where $\sigma$ only depends on time, see also \cite[Theorem 3.2.4(1)]{Wbook}. When $\mathbb{H} = \mathbb{R}^d$
is finite-dimensional, we may take $A = 0$ and so that Theorem \ref{T2.2} with $J^A = J $ covers
the result in \cite[Theorem 2.1]{W}. In this case, according to \cite{W}, the integration by parts formula
implies that $P_T$ has a density $ p_T(x, y)$ with respect to the Lebesgue measure, which is
differentiable in $y$ with
\beg{equation}\beg{split}\label{2.10}
\nabla_v\log p_T(x,\cdot)(y)=-\mathbb{E}\left(M^v_T|X_T(x)=y\right), \ \ x,v\in \mathbb{R}^d.
\end{split}\end{equation}

The next corollary is an application of Theorem \ref{T2.2} for finite dimension case.
\begin{cor}\label{C2.3} Assume (H1) and (H2), $\mathbb{H}=\mathbb{R}^d$, $A=0$. Let
\beg{equation*}\beg{split}
&\beta_1(p,t):=3^{p-1}\exp\left\{3^{p-1}\left(t^{p-1}+t^{\frac{p}{2}-1}\right)K^{p}_{1}(t)\right\},\\
&\beta_2(p,t):=3^{p-1}\exp\left\{3^{p-1} \left[t^{p-1}\left(K_{1}(t)+K^{2}_{1}(t)\right)^{p}+t^{\frac{p}{2}-1}K_{1}^{p}(t)\right]\right\}.
\end{split}\end{equation*}
Then the following assertions hold.
\beg{enumerate}
\item[(1)] For any $T>0$, $v\in\mathbb{R}^d$,
\begin{equation*}\beg{split}
&\|P_{T}(\nabla_v f)\|_{\infty}\leq |v|\|f\|_{\infty}\frac{\Gamma_T}{T},\ \ f\in C^1_b(\mathbb{R}^d),\\
&\int_{\mathbb{R}^d}|\nabla_v\log p_T(x,\cdot)|(y)p_T(x,y)\d y\leq |v|\frac{\Gamma_T}{T},\ \ x\in\mathbb{R}^d,
\end{split}\end{equation*}
where
\beg{equation*}\beg{split}
\Gamma_T&= \lambda(T)\sqrt{d T} \left(\beta_1(2,T)\beta_2(2,T)\right)^{\frac{1}{2}}\\
&+dT^2 K_2(T)\left(\beta_2(4,T)\beta_1(8,T)\beta_2^2(2,T)\right)^\frac{1}{4}\\
&+\left(d T^2 K_2(T)+d T \lambda_2(T)K_1(T)+d T^2 K_1(T) K_2(T)\right)\left(\beta_1(4,T)\beta_2(4,T)\right)^\frac{1}{2}.
\end{split}\end{equation*}
\item[(2)] For any $p > 1$, $T > 0$, it holds that
\begin{equation*}\beg{split}
&|P_{T}(\nabla_v f)|\leq \frac{|v|}{T}(P_{T}|f|^p)^{\frac{1}{p}}\left\{5^{[\frac{p}{p-1}]\vee2-1} \Gamma_{T,[\frac{p}{p-1}]\vee2}\right\}^{[\frac{p-1}{p}]\wedge\frac{1}{2}},\\
&\int_{\mathbb{R}^d}|\nabla_v\log p_T(x,\cdot)|^{\frac{p}{p-1}}(y)p_T(x,y)\d y\leq\frac{|v|}{T}\left\{5^{[\frac{p}{p-1}]\vee2-1} \Gamma_{T,[\frac{p}{p-1}]\vee2}\right\}^{[\frac{p-1}{p}]\wedge\frac{1}{2}},\ \ x\in\mathbb{R}^d,
\end{split}\end{equation*}
where
\beg{equation*}\beg{split}
\Gamma_{T,q}&= C(q)\lambda^q(T)d^{\frac{q}{2}} T^{\frac{q}{2}} \{\beta_1(2q,T)\beta_2(2q,T)\}^{\frac{1}{2}}\\
&+C(q)d^{q}  T^{\frac{3q+1}{2}} K_2^q(T)\left(\beta_2(4q,T)\beta_1(8q,T)\beta_2^2(2q,T)\right)^\frac{1}{4}\\
&+\left(d^q T^{2q} K^q_2(T)+d ^{q}T^{q} \lambda^q_2(T)K^q_1(T)+d ^{q}T^{2q}K^q_1(T) K^q_2(T)\right)\left(\beta_1(4q,T)\beta_2(4q,T)\right)^\frac{1}{2},
\end{split}\end{equation*}
and $C(q)$ is a nonnegetive constant only depending on $q\geq 2$.
\end{enumerate}
\end{cor}
\paragraph {Remark 2.3} From \eqref{2.1} and \eqref{2.3}, it is easy to see that $\mathbb{E}\left(\exp(\delta |M_T^v|)\right)=\infty$ for any $\delta>0$, for $M_T^v$ has the form like $\exp(\exp(X))$, where $X$ is a Gaussian random variable. Thus, it can not yield the shift Harnack inequality by Young's inequality from \eqref{2.8} as in the additive noise case.

The remainder of the paper is organized as follows. In Section 3, we give a proof of Theorem \ref{T2.2}, in Section 4, we prove Corollary 2.3.
\section{Proof of Theorem \ref{T2.2}}
Firstly, we introduce a formula for the solution of a class of semi-linear SDEs on $\L(\mathbb{H},\mathbb{H})$.
\begin{lem}\label{L3.1} Let $\{Y_{t}\}_{t\geq 0}$ solves the following SDE on $\L(\mathbb{H},\mathbb{H})$:
\begin{equation*}\beg{split}
\d Y_{t}=a_t\d t+b_tY_{t}\d t+\sum_{k=1}^{\infty}c^{k}_tY_{t}\d w^{k}_{t}+\sum_{k=1}^{\infty}f^{k}_t\d w^{k}_{t},
\end{split}\end{equation*}
where $a$, $b$, $\{c^{k}\}_{k\geq 1}$, $\{f^{k}\}_{k\geq1}$ are $\L(\mathbb{H},\mathbb{H})$-valued progressive measurable processes satisfying
\begin{equation*}\beg{split}
\int_0^t \mathbb{E}\left(\|a_s\|^{2}+\|b_s\|^{2}\right)\d s+\sum_{k=1}^{\infty}\int_0^t \mathbb{E}\left(\|c^{k}_s\|^{2}+\|f^{k}_s\|^{2}\right)\d s<\infty, \ \ t\geq 0.
\end{split}\end{equation*}
Then $Y_{t}$ satifies
\begin{equation}\label{3.1}\beg{split}
Y_{t}=G_t\left\{Y_{0}+\int_{0}^{t}G^{-1}_{s}a_s\d s+\int_{0}^{t}G^{-1}_{s}\sum_{k=1}^{\infty}f^{k}_s\d w^{k}_{s}-\int_{0}^{t}G^{-1}_{s}\sum_{k=1}^{\infty}c^{k}_sf^{k}_s\d s\right\},\ \ t\geq 0,
\end{split}\end{equation}
where $\{G_{t}\}_{t\geq 0}$ and $\{G_{t}^{-1}\}_{t\geq 0}$ satisfy
\begin{equation*}\beg{split}
&\d G_t=b_tG_t\d t+\sum_{k=1}^{\infty}c^{k}_tG_t\d w^{k}_{t}, \quad G_0=I;\\
&\d G_{t}^{-1}=-G_{t}^{-1}\left( b_t-\sum_{k=1}^{\infty}\left(c^{k}_t\right)^{2}\right)\d t-G_{t}^{-1}\sum_{k=1}^{\infty}c^{k}_t\d w^{k}_{t},\quad G_{0}^{-1}=I.
\end{split}\end{equation*}
\end{lem}
\begin{proof} Let $\{F_{t}\}_{t\geq 0}$ (called as integrating factor) solves the following SDE:
\begin{equation*}\beg{split}
\d F_{t}=-F_t b_t\d t-F_t\sum_{k=1}^{\infty}c^{k}_t\d w^{k}_{t}+F_{t}\sum_{k=1}^{\infty}\left(c^{k}_t\right)^{2}\d t,\quad F_{0}=I.
\end{split}\end{equation*}
Then It\^{o} formula yields
\begin{equation*}\beg{split}
\d F_{t}Y_{t}&=(\d F_{t})Y_{t}+F_{t}\d Y_{t}+(\d F_{t})(\d Y_{t})\\
&=F_{t}\left[-b_t\d t-\sum_{k=1}^{\infty}c^{k}_t\d w^{k}_{t}+\sum_{k=1}^{\infty}\left(c^{k}_t\right)^{2}\d t\right]Y_{t}\\
&+F_{t}\left[a_t\d t+b_tY_{t}\d t+\sum_{k=1}^{\infty}c^{k}_tY_{t}\d w^{k}_{t}+\sum_{k=1}^{\infty}f^{k}_t\d w^{k}_{t}\right]\\
&-F_{t}\sum_{k=1}^{\infty}\left(c^{k}_t\right)^{2}Y_{t}\d t-F_{t}\sum_{k=1}^{\infty}c^{k}_tf^{k}_t\d t\\
&=F_{t}a_t\d t+F_{t}\sum_{k=1}^{\infty}f^{k}_t\d w^{k}_{t}-F_{t}\sum_{k=1}^{\infty}c^{k}_tf^{k}_t\d t.
\end{split}\end{equation*}
So we obtain
\begin{equation}\label{3.2}\beg{split}
Y_{t}=F_{t}^{-1}\left\{F_{0}Y_{0}+\int_{0}^{t}F_{s}a_s\d s+\int_{0}^{t}F_{s}\sum_{k=1}^{\infty}f^{k}_s\d w^{k}_{s}-\int_{0}^{t}F_{s}\sum_{k=1}^{\infty}c^{k}_sf^{k}_s\d s\right\},
\end{split}\end{equation}
where $F_{t}^{-1}$ satisfies
\begin{equation}\beg{split}\label{3.3}
\d F_{t}^{-1}=b_tF_{t}^{-1}\d t+\sum_{k=1}^{\infty}c^{k}_tF_{t}^{-1}\d w^{k}_{t}, \quad F_{0}^{-1}=I.
\end{split}\end{equation}
Let $G_t:=F_{t}^{-1}$, we obtain \eqref{3.1}.
\end{proof}

\begin{proof} [Proof of Theorem 1.1] We will use Malliavin calculus to derive the integration by parts formula, see for instance \cite{N,W,W1}. For $( W_{t})_{t\in[0,T]}$, let $(D,\D(D))$ be the Malliavin gradient, and let $(D^{\ast},\D(D^{\ast}))$ be its adjoint operator (i.e. the Malliavin divergence). Since $(J_{s,t},J^{A} _{s,t})_{t\geq s}$ satisfy linear SDEs, from \eqref{2.1}, it is easy to see that
\beq\label{3.4}\beg{split}
J_{T}=J_{t,T}J_{t},\ \  J_{T}^{A}=J^{A} _{t,T}J_{t}^{A}, \ \ J_{t}^{A}=\e^{ At}J_{t},\quad T\geq t\geq 0.
\end{split}\end{equation}
Take
\begin{equation}\label{3.5}
h(t)=\int_{0}^{t}\sigma_{s}^{-1}( X_{s})J_{s}^{A}J_{T}^{-1}v\d s,\quad t\in[0,T].
\end{equation}
From (H1), we see that $J_{t}$ and $J_{t}^{-1}$ are Malliavin differentiable for every $t\in[0,T]$ such that $h\in\D(D^{\ast})$, so that \eqref{1.2} yields
\beq\label{3.6}\beg{split}
\d D_{h} X_{t}&=( A+\nabla  b_{t}( X_{t}))D_{h} X_{t}\d t\\
&+\sum_{k=1}^{\infty}\nabla\sigma^{(k)}_{t}(X_{t})D_{h} X_{t}\d w^{k}_{t}+\sigma_{t}( X_{t})\d h(t),\quad D_{h} X_{0}=0.
\end{split}\end{equation}
Then by Duhamel's formula, we obtain
\begin{equation*}\beg{split}
D_{h} X_{T}=\int_{0}^{T}J_{t,T}^{A}\sigma_{t}( X_{t})\d h(t)&=\int_{0}^{T}J_{t,T}^{A}J_{t}^{A}J_{T}^{-1}v\d t=T\e^{ AT}v.
\end{split}\end{equation*}
Therefore,
\beq\label{3.7}\beg{split}
\mathbb{E}(\nabla_{\e^{ AT}v}f)( X_{T})&=\frac{1}{T}\mathbb{E} (\nabla_{D_{h} X_{T}}f)( X_{T})=\frac{1}{T}\mathbb{E}\{D_{h}f( X_{T})\} =\frac{1}{T}\mathbb{E}\{f( X_{T})D^{*}(h)\}.
\end{split}\end{equation}
Since $h$ is not adpted, to calculate $D^{*}(h)$, let
\begin{equation*}\beg{split}
h_{k}(t)=\int_{0}^{t}\sigma_{s}^{-1}( X_{s})J_{s}^{A}e_{k}\d s, \quad F_{k}=\left\langle J_{T}^{-1}v,e_{k}\right\rangle, \ \ k\geq 1, t\in[0,T].
\end{split}\end{equation*}
Then $h(t)$ can be written as
\begin{equation*}\beg{split}
h(t)=\sum_{k=1}^{\infty}h_{k}(t)F_{k},\quad t\in[0,T].
\end{split}\end{equation*}
Noting that $h_{k}$ is adpted with $\mathbb{E}\int_{0}^{T}|h_{k}^{\prime}(t)|^{2}\d t<\infty$, we have
\begin{equation*}\beg{split}
D^{\ast}(h_{k})=\int_{0}^{T}\left\langle h_{k}^{\prime}(t),\d W_{t}\right\rangle=\int_{0}^{T}\left\langle\left(\sigma_{s}^{\ast}(\sigma_{s}\sigma_{s}^{\ast})^{-1} \right)( X_{s})J_{s}^{A}e_{k}, \d  W_{s}\right\rangle,\ \ k\geq 1.
\end{split}\end{equation*}
Thus, using the formula $D^{\ast}(F_{k}h_{k})=F_{k}D^{\ast}(h_{k})-D_{h_{k}}F_{k}$, we obtain
\beq\label{3.8}\beg{split}
D^{*}(h)&=\sum_{k=1}^{\infty}\{F_{k}D^{\ast}(h_{k})-D_{h_{k}}F_{k}\}\\
&=\left\langle\int_{0}^{T}\left[\left(\sigma_{s}^{\ast}(\sigma_{s}\sigma_{s}^{\ast})^{-1} \right)( X_{s})J_{s}^{A}\right]^{\ast}\d  W_{s},J_{T}^{-1}v\right\rangle-\sum_{k=1}^{\infty}\left\langle D_{h_{k}} J_{T}^{-1}v,e_{k}\right\rangle.
\end{split}\end{equation}
From \eqref{2.1} for $J_{t}:=J_{0,t}$, we have
\begin{equation*}\beg{split}
\d D_{h_{k}}J_{t}&= B_{t}(X_{t})D_{h_{k}}J_{t}\d t+\sum_{j=1}^{\infty}\Sigma^{(j)}_{t}( X_{t})D_{h_{k}}J_{t}\d  w^{j}_{t}\\
&+\left(\nabla_{D_{h_{k}} X_{t}} B_{t}\right)( X_{t})J_{t}\d t+\sum_{j=1}^{\infty}\left(\nabla_{D_{h_{k}} X_{t}}\Sigma^{(j)}_{t}\right)( X_{t})J_{t}\d  w^{j}_{t}\\
&+\sum_{j=1}^{\infty}\Sigma_{t}^{(j)}( X_{t})J_{t}\d h_{k}^{j}(t), \quad D_{h_{k}}J_{0}=0,
\end{split}\end{equation*}
where $h_{k}^{j}:=\langle h_{k},e_{j}\rangle$, $j\geq 1$. By Lemma 3.1, we obtain
\begin{equation}\beg{split}\label{3.9}
D_{h_{k}}J_{T}&=J_{T}\int_{0}^{T}J_{t}^{-1}
\left(\nabla_{D_{h_{k}} X_{t}} B_{t}\right)( X_{t})J_{t}\d t\\
&+J_{T}\int_{0}^{T}J_{t}^{-1}\sum_{j=1}^{\infty}\left(\nabla_{D_{h_{k}} X_{t}} \Sigma^{(j)}_{t}\right)( X_{t})J_{t}\d  w^{j}_{t}\\
&+J_{T}\int_{0}^{T}J_{t}^{-1}\sum_{j=1}^{\infty}\left\langle\left(\sigma_{t}^{\ast}(\sigma_{t}\sigma_{t}^{\ast})^{-1} \right)( X_{t})J_{t}^{A} e_{k},e_j\right\rangle\Sigma_{t}^{(j)}( X_{t})J_{t} \d t\\
&-J_{T}\int_{0}^{T}J_{t}^{-1}\sum_{j=1}^{\infty}\Sigma_{t}^{(j)}( X_{t}) \left(\nabla_{D_{h_{k}} X_{t}}\Sigma^{(j)}_{t}\right)( X_{t})J_{t}\d t.
\end{split}\end{equation}
In addition, it follows from \eqref{3.6} that
\begin{equation}\beg{split}\label{3.10}
D_{h_{k}} X_{t}=tJ_{t}^{A}e_{k}.
\end{split}\end{equation}
Since $0=D_{h_{k}}(J_{T}J_{T}^{-1})=J_{T} D_{h_{k}}J_{T}^{-1}+(D_{h_{k}}J_{T})J_{T}^{-1}$, it means that
\beq\label{3.11}\beg{split}
&D_{h_{k}}J_{T}^{-1}=-J_{T}^{-1}(D_{h_{k}}J_{T})J_{T}^{-1}.
\end{split}\end{equation}
Combining \eqref{3.8}, \eqref{3.9}, \eqref{3.10} and \eqref{3.11}, we get
\begin{equation*}\beg{split}
D^{\ast}(h)&=\left\langle\int_{0}^{T}\left[\sigma_{s}^{-1}( X_{s})J_{s}^{A}\right]^{\ast}\d  W_{s},J_{T}^{-1}v\right\rangle\\
&+\sum_{k=1}^{\infty}\int_{0}^{T}\left\langle J_{t}^{-1}
\left(\nabla_{tJ_{t}^{A}e_{k}} B_{t}\right)( X_{t})J_{t}J_{T}^{-1}v\d t, e_{k}\right\rangle\\
&+\sum_{k=1}^{\infty}\left\langle\left[\int_{0}^{T}J_{t}^{-1}\sum_{j=1}^{\infty}\left(\nabla_{tJ_{t}^{A}e_{k}} \Sigma^{(j)}_{t}\right)( X_{t})J_{t}\d  w^{j}_{t}\right]J_{T}^{-1}v, e_{k}\right\rangle\\
&+\sum_{k=1}^{\infty}\int_{0}^{T}\sum_{j=1}^{\infty}\left\langle\sigma_{t}^{-1}( X_{t})J_{t}^{A} e_{k},e_j\right\rangle\left\langle J_{t}^{-1}\Sigma_{t}^{(j)}( X_{t})J_{t}J_{T}^{-1}v\d t, e_{k}\right\rangle\\
&-\sum_{k=1}^{\infty}\int_{0}^{T}\left\langle J_{t}^{-1}\sum_{j=1}^{\infty}\Sigma_{t}^{(j)}( X_{t}) \left(\nabla_{tJ_{t}^{A}e_{k}}\Sigma^{(j)}_{t}\right)( X_{t})J_{t}J_{T}^{-1}v\d t, e_{k}\right\rangle\\
&=\left\langle\int_{0}^{T}\left[\sigma^{-1}_t(X_{t})J^{A}_{t}\right]^{\ast}\d W_{t},J_{T}^{-1}v\right\rangle\\
&+\int_{0}^{T}\mathrm{Tr}\left\{t J_{t}^{-1}\left[
(\nabla B_{t})(X_{t})\left(J_{t}J_{T}^{-1}v,J^{A}_{t}\right)\right]\right\}\d t\\
&+\sum_{k=1}^{\infty}\left\langle\left[\int_{0}^{T}J_{t}^{-1}\sum_{j=1}^{\infty}\left(\nabla_{tJ_{t}^{A}e_{k}} \Sigma^{(j)}_{t}\right)(X_{t})J_{t}\d w^{j}_{t}\right]J_{T}^{-1}v, e_{k}\right\rangle\\
&+\int_{0}^{T}\mathrm{Tr}\left\{J_{t}^{-1}\left[\Sigma_{t}(X_{t})J_{t}J_{T}^{-1}v\right]\left(\sigma_t^{\ast} (\sigma_{t}\sigma_t^{\ast})^{-1}\right)(X_{t})J^{A}_{t} \right\}\d t\\
&-\int_{0}^{T}\mathrm{Tr}\left\{tJ_{t}^{-1}\sum_{j=1}^{\infty}\Sigma^{(j)}_{t}(X_{t})\left[\left(\nabla \Sigma^{(j)}_{t}\right)(X_{t})\left(J_{t}J_{T}^{-1}v,J^{A}_{t}\right)\right]\right\}\d t
\end{split}\end{equation*}
Substituting this into \eqref{3.7}, we obtain \eqref{2.8}.
\end{proof}
\section{Proof of Corollary \ref{C2.3}}
\begin{proof}[Proof of Corollary 2.3] (1) For simplicity, letting
\beg{equation*}\beg{split}
&\Theta_{1}=\left\langle\int_{0}^{T}\left[\sigma^{-1}_t(X_{t})J_{t}\right]^{\ast}\d W_{t},J_{T}^{-1}v\right\rangle;\\
&\Theta_{2}=\int_{0}^{T}\mathrm{Tr}\left\{t J_{t}^{-1}\left[
(\nabla B_{t})(X_{t})\left(J_{t}J_{T}^{-1}v,J_{t}\right)\right]\right\}\d t;\\
&\Theta_{3}=\sum_{k=1}^{d}\left\langle\left[\int_{0}^{T}J_{t}^{-1}\sum_{j=1}^{d}\left(\nabla_{tJ_{t}e_{k}} \Sigma^{(j)}_{t}\right)(X_{t})J_{t}\d w^{j}_{t}\right]J_{T}^{-1}v, e_{k}\right\rangle;\\
&\Theta_{4}=\int_{0}^{T}\mathrm{Tr}\left\{J_{t}^{-1}\left[\Sigma_{t}(X_{t})J_{t}J_{T}^{-1}v\right]\sigma_t^{-1} (X_{t})J_{t} \right\}\d t;\\
&\Theta_{5}=-\int_{0}^{T}\mathrm{Tr}\left\{tJ_{t}^{-1}\sum_{j=1}^{\infty}\Sigma^{(j)}_{t}(X_{t})\left[\left(\nabla \Sigma^{(j)}_{t}\right)(X_{t})\left(J_{t}J_{T}^{-1}v,J_{t}\right)\right]\right\}\d t,
\end{split}\end{equation*}
then $M^v_T=\sum_{i=1}^5\Theta_i$.

Firstly, by H\"{o}lder inequality, It\^{o} isometric formula, (H2), Lemma \ref{L2.1}, we have
\beg{equation}\beg{split}\label{4.1}
\mathbb{E}|\Theta_{1}|&\leq \left\{\int_{0}^{T}\mathbb{E}\|\sigma^{-1}_t(X_{t})J_{t}\|^2_{\mathrm{HS}}\d t\mathbb{E}\left|J_{T}^{-1}v\right|^2\right\}^{\frac{1}{2}}\\
&\leq \lambda(T)\sqrt{d T} |v|\left(\beta_1(2,T)\beta_2(2,T)\right)^{\frac{1}{2}}.
\end{split}\end{equation}
Next, H\"{o}lder inequality, (H1), Lemma \ref{L2.1} yield that
\beg{equation}\beg{split}\label{4.2}
\mathbb{E}|\Theta_{2}|&\leq d T K_2(T)|v|\int_{0}^{T}\left\{\mathbb{E}\left\|J_{t}\right\|^4\right\}^{\frac{1}{2}} \left\{\mathbb{E}\left\|J_{t}^{-1}\right\|^2\left\|J_{T}^{-1}\right\|^2\right\}^{\frac{1}{2}}\d t\\
&\leq d T K_2(T)|v|\int_{0}^{T}\left\{\mathbb{E}\left\|J_{t}\right\|^4\right\}^{\frac{1}{2}} \left\{\mathbb{E}\left\|J_{t}^{-1}\right\|^4\right\}^{\frac{1}{4}} \left\{\mathbb{E}\left\|J_{T}^{-1}\right\|^4\right\}^{\frac{1}{4}}\d t\\
&\leq d T^2 K_2(T)|v|\left(\beta_1(4,T)\beta_2(4,T)\right)^\frac{1}{2}.
\end{split}\end{equation}
Again by H\"{o}lder inequality, It\^{o} isometric formula, (H1), Lemma \ref{L2.1}, it holds that
\beg{equation}\beg{split}\label{4.3}
\mathbb{E}|\Theta_{3}|&\leq dT K_2(T)|v|\int_{0}^{T}\left\{\mathbb{E}\left(\left\|J_{t}^{-1} \right\|^2\left\|J_{t}\right\|^4\right)\right\}^\frac{1}{2}\left\{\mathbb{E}\left \|J_{T}^{-1}\right\|^2\right\}^\frac{1}{2}\d  t\\
&\leq dT K_2(T)|v|\int_{0}^{T}\left\{\mathbb{E}\left\|J_{t}^{-1} \right\|^4\right\}^\frac{1}{4}\left\{\mathbb{E}\left\|J_{t} \right\|^8\right\}^\frac{1}{4}\left\{\mathbb{E}\left \|J_{T}^{-1}\right\|^2\right\}^\frac{1}{2}\d  t\\
&\leq dT^2 K_2(T)|v|\left(\beta_2(4,T)\beta_1(8,T)\beta_2^2(2,T)\right)^\frac{1}{4}.
\end{split}\end{equation}
Similarly as \eqref{4.2}, it is easy to see that
\beg{equation}\beg{split}\label{4.4}
\mathbb{E}|\Theta_{4}|&\leq d \lambda_2(T)K_1(T)|v|\int_{0}^{T}\left\{\mathbb{E}\left\|J_{t}\right\|^4\right\}^{\frac{1}{2}} \left\{\mathbb{E}\left\|J_{t}^{-1}\right\|^2\left\|J_{T}^{-1}\right\|^2\right\}^{\frac{1}{2}}\d t\\
&d T \lambda_2(T)K_1(T)|v|\left(\beta_1(4,T)\beta_2(4,T)\right)^\frac{1}{2},
\end{split}\end{equation}
and
\beg{equation}\beg{split}\label{4.5}
\mathbb{E}|\Theta_{5}|&\leq d T K_1(T) K_2(T)|v|\int_{0}^{T}\left\{\mathbb{E}\left\|J_{t}\right\|^4\right\}^{\frac{1}{2}} \left\{\mathbb{E}\left\|J_{t}^{-1}\right\|^2\left\|J_{T}^{-1}\right\|^2\right\}^{\frac{1}{2}}\d t\\
&\leq d T^2 K_1(T) K_2(T)|v|\left(\beta_1(4,T)\beta_2(4,T)\right)^\frac{1}{2}.
\end{split}\end{equation}
Combining \eqref{4.1}-\eqref{4.5}, we have
\beg{equation*}\beg{split}
\mathbb{E}|M^v_T|\leq |v|\Gamma_T,
\end{split}\end{equation*}
here, $\Gamma_T$ is in Corollary \ref{C2.3}. According to \eqref{2.8}, it is easy to see that
\beg{equation*}\beg{split}
|P_{T}(\nabla_{v}f)|\leq\frac{\|f\|_{\infty}}{T}\mathbb{E}|M^v_T|\leq |v|\|f\|_{\infty}\frac{\Gamma_T}{T},\ \ f\in C_b^1{\mathbb{R}^d}.
\end{split}\end{equation*}
By \cite[Theorem 2.4(1)]{W} with $H(r) = r$, \eqref{2.10} yields that
\begin{equation*}\beg{split}
\int_{\mathbb{R}^d}|\nabla_v\log p_T(x,\cdot)|(y)p_T(x,y)\d y\leq |v|\frac{\Gamma_T}{T},\ \ x\in\mathbb{R}^d,
\end{split}\end{equation*}
(2) For any $q\geq 2$, by Burkerholder-Davis-Gundy inequality, H\"{o}lder inequality, (H2), Lemma \ref{L2.1}, we have
\beg{equation}\beg{split}\label{4.6}
\mathbb{E}|\Theta_{1}|^q&\leq C(q)\lambda^q(T)T^{\frac{q-1}{2}} |v|^q \left\{\int_{0}^{T}\mathbb{E}\|J_{t}\|^{2q}_{\mathrm{HS}}\d t\right\}^\frac{1}{2}\left\{\mathbb{E}\left|J_{T}^{-1}\right|^{2q}\right\}^\frac{1}{2}\\
&\leq C(q)\lambda^q(T)d^{\frac{q}{2}} T^{\frac{q}{2}} |v|^q\{\beta_1(2q,T)\beta_2(2q,T)\}^{\frac{1}{2}}.
\end{split}\end{equation}
Next, H\"{o}lder inequality, (H1), Lemma \ref{L2.1} yield that
\beg{equation}\beg{split}\label{4.7}
\mathbb{E}|\Theta_{2}|^q&\leq d^{q-1}T^{q-1}d T^q K^q_2(T)|v|^q\int_{0}^{T}\left\{\mathbb{E}\left\|J_{t}\right\|^{4q}\right\}^{\frac{1}{2}} \left\{\mathbb{E}\left\|J_{t}^{-1}\right\|^{2q}\left\|J_{T}^{-1}\right\|^{2q}\right\}^{\frac{1}{2}}\d t\\
&\leq d^{q-1}T^{q-1}d T^q K^q_2(T)|v|^q\int_{0}^{T}\left\{\mathbb{E}\left\|J_{t}\right\|^{4q}\right\}^{\frac{1}{2}} \left\{\mathbb{E}\left\|J_{t}^{-1}\right\|^{4q}\right\}^{\frac{1}{4}} \left\{\mathbb{E}\left\|J_{T}^{-1}\right\|^{q}\right\}^{\frac{1}{4}}\d t\\
&\leq d^q T^{2q} K^q_2(T)|v|^q\left(\beta_1(4q,T)\beta_2(4q,T)\right)^\frac{1}{2}.
\end{split}\end{equation}
Again by Burkerholder-Davis-Gundy inequality, H\"{o}lder inequality, (H1), Lemma \ref{L2.1}, it holds that
\beg{equation}\beg{split}\label{4.8}
\mathbb{E}|\Theta_{3}|^q&\leq C(q)d^{q}  T^{\frac{3q-1}{2}} K_2^q(T)|v|^q\int_{0}^{T}\left\{\mathbb{E}\left(\left\|J_{t}^{-1} \right\|^{2q}\left\|J_{t}\right\|^{4q}\right)\right\}^\frac{1}{2}\left\{\mathbb{E}\left \|J_{T}^{-1}\right\|^{2q}\right\}^\frac{1}{2}\d  t\\
&\leq C(q)d^{q}  T^{\frac{3q-1}{2}} K_2^q(T)|v|^q\int_{0}^{T}\left\{\mathbb{E}\left\|J_{t}^{-1} \right\|^{4q}\right\}^\frac{1}{4}\left\{\mathbb{E}\left\|J_{t} \right\|^{8q}\right\}^\frac{1}{4}\left\{\mathbb{E}\left \|J_{T}^{-1}\right\|^{2q}\right\}^\frac{1}{2}\d  t\\
&\leq C(q)d^{q}  T^{\frac{3q+1}{2}} K_2^q(T)|v|^q\left(\beta_2(4q,T)\beta_1(8q,T)\beta_2^2(2q,T)\right)^\frac{1}{4}.
\end{split}\end{equation}
Similarly as \eqref{4.7}, it is easy to see that
\beg{equation}\beg{split}\label{4.9}
\mathbb{E}|\Theta_{4}|^q&\leq d ^{q-1}T^{q-1} d \lambda^q_2(T)K^q_1(T)|v|^q\int_{0}^{T}\left\{\mathbb{E}\left\|J_{t}\right\|^{4q}\right\}^{\frac{1}{2}} \left\{\mathbb{E}\left\|J_{t}^{-1}\right\|^{2q}\left\|J_{T}^{-1}\right\|^{2q}\right\}^{\frac{1}{2}}\d t\\
&\leq d ^{q}T^{q} \lambda^q_2(T)K^q_1(T)|v|^q\left(\beta_1(4q,T)\beta_2(4q,T)\right)^\frac{1}{2},
\end{split}\end{equation}
and
\beg{equation}\beg{split}\label{4.10}
\mathbb{E}|\Theta_{5}|^q&\leq d ^{q-1}T^{q-1}d K^q_1(T)T^q K^q_2(T)|v|^q\int_{0}^{T}\left\{\mathbb{E}\left\|J_{t}\right\|^{4q}\right\}^{\frac{1}{2}} \left\{\mathbb{E}\left\|J_{t}^{-1}\right\|^{2q}\left\|J_{T}^{-1}\right\|^{2q}\right\}^{\frac{1}{2}}\d t\\
&\leq d ^{q}T^{2q}K^q_1(T) K^q_2(T)|v|^q\left(\beta_1(4q,T)\beta_2(4q,T)\right)^\frac{1}{2}.
\end{split}\end{equation}
Combining \eqref{4.6}-\eqref{4.10}, for any $q\geq 2$, it holds that
\begin{equation}\beg{split}\label{4.11}
\left(\mathbb{E}|M^v_T|^{q}\right)^{\frac{1}{q}}\leq \left\{5^{q-1}\sum_{i=1}^5\mathbb{E}|\Theta_i|^q\right\}^{\frac{1}{q}}|v|\leq \left\{5^{q-1}\Gamma_{T,q}\right\}^{\frac{1}{q}}|v|,
\end{split}\end{equation}
here, $\Gamma_{T,q}$ is in Corollary \ref{C2.3}. On the other hand, Jensen inequality yields that
\begin{equation}\beg{split}\label{4.12}
\left(\mathbb{E}|M^v_T|^{q}\right)^{\frac{1}{q}}\leq \left(\mathbb{E}|M^v_T|^{2}\right)^{\frac{1}{2}}
\end{split}\end{equation}
for any $1<q<2$. Combining \eqref{4.11} and \eqref{4.12}, we obtain for any $q>1$,
\begin{equation}\beg{split}\label{4.13}
\left(\mathbb{E}|M^v_T|^{q}\right)^{\frac{1}{q}}\leq \left\{5^{q\vee 2-1}\Gamma_{T,q\vee2}\right\}^{\frac{1}{q\vee2}}|v|.
\end{split}\end{equation}
It follows from \eqref{2.8}, H\"{o}lder inequality and \eqref{4.13} that for any $p>1$,
\beg{equation*}\beg{split}
|P_{T}(\nabla_v f)|&\leq\frac{1}{T}(P_{T}|f|^p)^{\frac{1}{p}}\left(\mathbb{E}|M^v_T|^{\frac{p}{p-1}}\right)^{\frac{p-1}{p}}\\
&\leq \frac{|v|}{T}(P_{T}|f|^p)^{\frac{1}{p}}\left\{5^{[\frac{p}{p-1}]\vee2-1} \Gamma_{T,[\frac{p}{p-1}]\vee2}\right\}^{[\frac{p-1}{p}]\wedge\frac{1}{2}}, \ \ f\in C_b^1{\mathbb{R}^d}.
\end{split}\end{equation*}
Finally, applying \cite[Theorem 2.4(1)]{W} with $H(r) = r^{\frac{p}{p-1}}$, \eqref{2.10} yields that
\begin{equation*}\beg{split}
\int_{\mathbb{R}^d}|\nabla_v\log p_T(x,\cdot)|^{\frac{p}{p-1}}(y)p_T(x,y)\d y\leq\frac{|v|}{T}\left\{5^{[\frac{p}{p-1}]\vee2-1} \Gamma_{T,[\frac{p}{p-1}]\vee2}\right\}^{[\frac{p-1}{p}]\wedge\frac{1}{2}},\ \ x\in\mathbb{R}^d.
\end{split}\end{equation*}
Thus, the proof is completed.
\end{proof}

\paragraph{Acknowledgement.} The authors would like to thank Profeesor Feng-Yu Wang for  corrections and helpful comments.

\beg{thebibliography}{99}

\bibitem{D}  B. Driver, \emph{Integration by parts for heat kernel measures revisited,} J. Math. Pures
Appl. 76(1997), 703-737.

\bibitem{F} X.-L. Fan, \emph{Integration by parts formula, derivative formula, and transportation
inequalities for SDEs driven by fractional Brownian motion, } Stoch. Anal. Appl. 33(2015), 199-212.

\bibitem{H} Z.-Y. Huang, \emph{Basis for Stochastic Analysis (in Chinese), } Science press, Beijing, 2001.

\bibitem{J} N. Jacob, \emph{Pseudo Differential Operators and Markov Processes (Volume I),} Imperial
College Press, London, 2001.

\bibitem{M} P. Malliavin, \emph{Stochastic analysis, } Springer-Verlag, Berlin, 1997.

\bibitem{N} D. Nualart, \emph{The Malliavin calculus and related topics,} Second Edition, Springer-
Verlag, Berlin, 2005.

\bibitem{W} F.-Y. Wang, \emph{Integration by parts formula and shift Harnack inequality for stochastic
equations,} Ann. Probab. 42(2014), 994-1019.

\bibitem{W1} F.-Y. Wang, \emph{Integration by parts formula and applications for SDEs with L\'{e}vy noise
(in Chinese),} Sci. Sin. Math. 45(2015), 461-470.

\bibitem{W2} F.-Y. Wang, \emph{Integration by Parts Formula and Applications for SPDEs with Jumps,} Stochastics. 88(2016), 737-750.

\bibitem{Wbook} F.-Y. Wang, \emph{Harnack Inequality and Applications for Stochastic Partial Differential Equations,} Springer, New York, 2013.
    
\bibitem{SQZ} S.-Q. Zhang,  \emph{Shift Harnack Inequality and Integration by Part Formula for Semilinear SPDE,} Front. Math. China 11(2016), 461-496.

\end{thebibliography}

\end{document}